\documentclass[12pt,reqno]{amsart}

\usepackage{amssymb,amscd,amsthm,latexsym}

\usepackage[all]{xy}
 \usepackage{color}
\theoremstyle{plain}
\newtheorem{theorem}[subsection]{Theorem}

\newtheorem{proposition}[subsection]{Proposition}
\newtheorem{corollary}[subsection]{Corollary}

\theoremstyle{definition}

\newtheorem{remark}[subsection]{Remark}

\newdir{ >}{% end of arrow for the monos
  @{}*!/-10pt/@{>} }

\newcommand{\CC}{ \ensuremath{\mathcal {C}} }
\newcommand{\XX}{ \ensuremath{\mathcal {X}} }

\newcommand{\Eq}{ \ensuremath{\mathrm{Eq}} }
\newcommand{\Pt}{ \ensuremath{\mathrm{Pt}} }

\newcommand{\map}[2]{ \ensuremath{ \xymatrix@1@C=15pt{ #1 \ar[r] & #2 } } }
\newcommand{\mono}[2]{ \ensuremath{ \xymatrix@1@C=15pt{ #1 \ar@{ >->}[r] & #2 } } }
\newcommand{\regepi}[2]{ \ensuremath{ \xymatrix@1@C=15pt{ #1 \ar@{>>}[r] & #2 } } }

\begin{document}
\title[Beck-Chevalley condition and Goursat categories]%
{Beck-Chevalley condition and Goursat categories }

\author{ Marino Gran}
\address{Insitut de Recherche en Math\'ematique et Physique, Universit\'e Catholique de Louvain, Chemin du Cyclotron 2,
1348 Louvain-la-Neuve, Belgium}

\email{marino.gran@uclouvain.be}

\author[Diana Rodelo]{Diana Rodelo}
\address{Departamento de Matem\'atica, Faculdade de Ci\^{e}ncias e Tecnologia, Universidade
do Algarve, Campus de Gambelas, 8005-139 Faro and Centre for Mathematics of the University of Coimbra, 3001-454 Coimbra, Portugal}

\email{drodelo@ualg.pt}
\thanks{This work was partially supported by the Centre for Mathematics of the University of
Coimbra -- UID/MAT/00324/2013, funded by the Portuguese Government through FCT/MEC and
co-funded by the European Regional Development Fund through the Partnership Agreement
PT2020.}

\keywords{Goursat categories, $3$-permutable varieties, Shifting Lemma, Beck-Chevalley condition, reflective subcategories, Galois groupoid.}

\subjclass[2000]{
08C05, %Categories of algebras
08B05, % Categories admitting limits (complete categories), functors preserving limits, completions
18C05,
18B99, % Special categories
18E10} % Exact categories, abelian categories

%%%%%%%%%%%%%%%%%%%%%%%%%%%%%%%%%%%%%%%  ABSTRACT  %%%%%%%%%%%%%%%%%%%%%%%%%%%%%%%%%%%%%%%%%%%%%%%%%%%%%%%%%%%%%%%%%%%%
\begin {abstract}
We characterise regular Goursat categories through a specific stability property of regular epimorphisms with respect to pullbacks. Under the assumption of the existence of some pushouts this property can be also expressed as a restricted Beck-Chevalley condition, with respect to the fibration of points, for a special class of commutative squares. In the case of varieties of universal algebras these results give, in particular, a structural explanation of the existence of the ternary operations characterising $3$-permutable varieties of universal algebras.
\end {abstract}

%\infonum{?}{15}

\date{\today}

\maketitle

%%%%%%%%%%%%%%%%%%%%%%%%%%%%%%%%%%%%%%%  INTRODUCTION  %%%%%%%%%%%%%%%%%%%%%%%%%%%%%%%%%%%%%%%%%%%%%%%%%%%%%%%%%%%%%%%%

A variety of universal algebras is called a Mal'tsev variety \cite{Smith} when any pair of congruences $R$ and $S$ on the same algebra $2$-permute, meaning that $RS=SR$. The celebrated Mal'tsev theorem asserts that the algebraic theory of such a variety is characterised by the existence of a ternary term $p(x,y,z)$ such that the identities $p(x,y,y)=x$ and $p(x,x,y)=y$ hold~\cite{Maltsev-Sbornik}. The weaker $3$-permutability of congruences $RSR=SRS$, which defines $3$-permutable varieties, is also equivalent to the existence of two ternary operations $r$ and $s$ such that the identities $r(x,y,y)=x$, $r(x,x,y)=s(x,y,y)$ and $s(x,x,y)=y$ hold~\cite{Hagemann-Mitschke}. A nice feature of $3$-permutable varieties is the fact that they are congruence modular, a condition that plays a crucial role in the development of commutator theory \cite{FM, Gumm}.

Many interesting results have been discovered in regular Mal'tsev categories \cite{CLP} and in regular Goursat categories \cite{CKP}, which can be seen as the categorical extensions of Mal'tsev varieties and of $3$-permutable varieties, respectively. The interested reader will find many properties of these categories in the references  ~\cite{BerBour, BB, MCFPO, BR, CKP, CLP, GR3x3, GRCuboid, JK, JRVdL, L, Pedicchio}, for instance.

For Mal'tsev categories, and many other algebraic categories, there are some elegant characterisations expressed in terms of the fibration of points~\cite{MCFPO, BB}. Recall that, given a category $\CC$ with pullbacks, the fibration of points is the functor $\Pt (\CC) \to \CC$ associating the codomain $W$ to any ``point'' in $\mathcal C$, i.e. to any split epimorphism $g \colon U \to W$ with a given splitting $j \colon W \to U$. For any morphism $\beta\colon Y \to W$ in $\mathcal C$, the change-of-base functor, defined by pulling back along $\beta$, is denoted by $\beta^*\colon \Pt_{W}(\CC) \to \Pt_Y(\CC)$. One of the goals of this paper is to give a characterisation of regular Goursat categories by using this fibration. It turns out that such a characterisation not only involves the change-of-base functors $\beta^*$ with respect to the fibration of points, but also their left adjoints $\beta_!$, which exist as soon as the category admits pushouts of split monomorphisms~\cite{NEKEAC}. It is precisely the so-called Beck-Chevalley condition (Theorem~\ref{corollary Goursat cube}) for the commutative squares of the following type
$$
\vcenter{\xymatrix{
  X \ar@{>>}[r]^-{\alpha} \ar@<-2pt>[d]_-f & U  \ar@<-2pt>[d]_-g \\
  Y \ar@{>>}[r]_-{\beta} \ar@<-2pt>[u]_-i & W, \ar@<-2pt>[u]_-j }}
$$
where $\alpha$ and $\beta$ are regular epimorphisms and $f$ and $g$ are split epimorphisms (i.e. the pair $(\alpha, \beta)$ is a regular epimorphism in the category $\Pt (\CC)$).
In fact, the Goursat property can be also expressed in terms of the functors $\beta_!$ alone: $\beta_!$ preserves binary products, for any regular epimorphism $\beta$.

The proof of these results also relies on the fact that the Shifting Lemma~\cite{Gumm} holds in any regular Goursat category~\cite{BG0}, since the lattice of equivalence relations on any object is modular~\cite{CKP}. A more general characterisation of Goursat categories among regular categories is also obtained without requiring the existence of pushouts along split monomorphisms and involves a stability property for regular epimorphisms (Theorem \ref{Goursat cube}).
In the varietal case, the existence of the ternary operations characterising $3$-permutable varieties mentioned above can be deduced from this theorem by applying it to a suitable diagram involving free algebras (Remark~\ref{ternary}).

As a consequence of the results in this paper, we obtain an extension to regular Goursat categories of a known result in the regular Mal'tsev context (Proposition 3.6 in~\cite{BR}): the reflector to a (regular epi)-reflective subcategory of a regular Goursat category always preserves pullbacks of split epimorphisms along split epimorphisms (Proposition~~\ref{I preserves pbs}). It then follows that the so-called internal Galois pregroupoid~\cite{GJ} associated to an extension is necessarily an internal groupoid (Corollary~\ref{Galois groupoid}).

Let us finally mention that the possibility of formulating some exactness properties of Mal'tsev categories in terms of a suitable Beck-Chevalley condition has been first suggested to the authors by Zurab Janelidze during the \emph{Workshop in Category Theory} at the University of Coimbra in 2012, where the work
~\cite{GRCuboid} was presented. We would like to warmly thank Zurab for this suggestion. The present paper shows that a suitable Beck-Chevalley condition characterises Goursat categories. Independently, Clemens Berger and Dominique Bourn have discovered a stronger condition that holds in the special case of exact Mal'tsev categories (Proposition $1.24$ in \cite{BerBour}). \\ We would like to thank Francis Borceux for some useful remarks concerning the proof of Theorem \ref{Goursat cube}. 
%%%%%%%%%%%%%%%%%%%%%%%%%%%%%%%%%%%%%%%  SECTION: GOURSAT CATEGORIES  %%%%%%%%%%%%%%%%%%%%%%%%%%%%%%%%%%%%%%%%%%%%%%%%%

\section{Goursat categories}
In this section we give a new characterisation for a regular category~\cite{EC} to be a Goursat category through a stability property of regular epimorphisms, similar to the known characterisation for regular Mal'tsev categories given in Proposition 3.6 in~\cite{GRCuboid}.

Recall that a regular category $\CC$ is called a \emph{Goursat} category~\cite{CKP} when the composition of (effective) equivalence relations $R$ and $S$ on a same object in $\CC$ is $3$-permutable: $RSR=SRS$. Given a relation $R=(R, r_1,r_2)$ from an object $X$ to an object $Y$, we write $R^o$ for the opposite relation $(R,r_2,r_1)$ from $Y$ to $X$.

From~\cite{CKP} and~\cite{JRVdL} we have:

\begin{theorem}
\label{Goursat characterisations} Let $\CC$ be a regular category. The following conditions are equivalent:
\begin{enumerate}
    \item[(i)] $\CC$ is a Goursat category;
    \item[(ii)] $E^{\circ}\leqslant EE$, for any reflexive relation $E$;
    \item[(iii)] $(1_X\wedge T)T^{\circ}(1_X\wedge T)\leqslant TT$, for any relation $T$ on an object $X$;
    \item[(iv)] $PP^{\circ}PP^{\circ}\leqslant PP^{\circ}$, for any relation $P$.
\end{enumerate}
\end{theorem}
\begin{proof}
(i) $\Leftrightarrow$ (ii). By Theorem 1 in~\cite{JRVdL}.\\
(i) $\Leftrightarrow$ (iv). By Theorem 3.5 of~\cite{CKP}.\\
(i) $\Leftrightarrow$ (iii). This type of equivalence was mentioned at the end of~\cite{JRVdL}. Note that condition (iii) is a stronger version of condition (ii), so that it will suffice to prove that (iv) $\Rightarrow$ (iii). In a regular context, it suffices to give a proof in set-theoretical terms (see Metatheorem $A.5.7$ in \cite{BB}, for instance). Suppose that $(x,x) \in T, (y,x)\in T$ and $(y,y)\in T$. We want to prove that $(x,y)\in TT$, i.e. that $(x,\alpha) \in T$ and $(\alpha, y) \in T$, for some $\alpha \in X$. We define a relation $P$ from $X\times X$ to $X$ by: $((a,b),c)\in P$ if and only if $(a,c)\in T$ and $(c,b)\in T$. One then sees that
$$ (x,x)Px, (y,x)Px, (y,x)Py, (y,y)Py$$
implying that
 $$  (x,x) PP^{\circ}PP^{\circ} (y,y).$$
  By condition  (iv) it follows that
   $$ (x,x) PP^{\circ} (y,y)$$
 and there is an $\alpha \in X$ such that $$(x,x) P \alpha, (y,y)P\alpha.$$ One concludes then that
 $$xT\alpha T y.$$
\end{proof}

Before proving the main results of this section we need to recall a useful property of regular Goursat categories, namely the validity of the so-called \emph{Shifting Lemma} \cite{Gumm}. In the context of varieties of universal algebras this property is equivalent to the modularity of the lattice of congruences on any of its algebras. The modularity of the lattice $(L_X, \vee, \wedge)$ of equivalence relations on any object $X$ also holds in any regular Goursat category, as shown in \cite{CKP}. More precisely, given equivalence relations $R, S$ and $T$ in $L_X$,
$$
    R\leqslant T \Rightarrow \left(\; R\vee (S\wedge T)=(R\vee S)\wedge T \;\right).
$$
 In this context, the supremum $R \vee S$ of two equivalence relations in $L_X$ is given by the triple relational composite $R \vee S = RSR$. By using generalized elements the validity of the Shifting Lemma can be expressed as follows:

\noindent \textbf{Shifting Lemma} \\
Given equivalence relations $R,S$ and $T$ on the same object $X$ such that $R\wedge S\leqslant T$, whenever $x,y,t,z$ are elements in $X$ with $(x,y)\in R \wedge T$, $(x,t) \in S$, $(y,z)\in S$ and $(t,z)\in R$, it then follows that $(t,z) \in T$:
\begin{equation}\label{shifting}
  \vcenter{\xymatrix@C=30pt{
    x \ar@{-}[r]^-S \ar@{-}[d]^-R \ar@(l,l)@{-}[d]_-T & t \ar@{-}[d]_-R \ar@(r,r)@{--}[d]^-T \\
    y \ar@{-}[r]_-S  & z }}
\end{equation}

The Shifting Lemma holds in any regular Goursat category~\cite{BG0}, as we are now going to recall by using the internal logic of a regular category. Given a diagram \eqref{shifting}, the $3$-permutability of the equivalence relations implies that $$(t,z) \in S (R \wedge T) S = (R\wedge T) S (R\wedge T).$$ Accordingly, there exist $a$ and $b$ such that $(t,a) \in R \wedge T$, $(a,b) \in S$, and $(b, z) \in R \wedge T$. Then $(a,b)$ is also in $R$, thus in $T$, since $R \wedge S \leqslant T$; it follows that $(t,z) \in T$, as desired.

The property expressed by the Shifting Lemma has been extended to a categorical context in~\cite{BG1}, giving rise to the notion of Gumm category. Indeed, the Shifting Lemma can be equivalently reformulated in any finitely complete category $\mathcal C$ by asking that a specific class of internal functors are discrete fibrations (see \cite{BG0} and \cite{BG1} for more details).

One of the fundamental results in this paper is given in Theorem~\ref{Goursat cube} below, where regular Goursat categories are characterised by a stability property of regular epimorphisms with respect to pullbacks. Such a stability condition is an extension of the following one where regular epimorphisms are stable with respect to kernel pairs:

\begin{theorem}~\cite{GR3x3}
\label{Goursat po}
Let $\CC$ be a regular category. The following conditions are equivalent:
\begin{enumerate}
\item[$(i)$] $\CC$ is a Goursat category;
\item[$(ii)$] any commutative diagram
$$\xymatrix{
  X \ar@{>>}[r]^-{\alpha} \ar@<-2pt>[d]_-f & U  \ar@<-2pt>[d]_-g \\
  Y \ar@{>>}[r]_-{\beta} \ar@<-2pt>[u]_-i & W, \ar@<-2pt>[u]_-j }
$$
 where $f$ and $g$ are split epimorphisms and $\alpha$ and $\beta$ are regular epimorphisms (commuting also with the splittings), is a \emph{Goursat pushout}: the comparison morphism $\lambda\colon \Eq(f) \to \Eq(g)$ induced by the universal property of the kernel pair $\Eq(g)$ of $g$ is also a regular epimorphism.
\end{enumerate}
\end{theorem}

\begin{theorem}
\label{Goursat cube} Let $\CC$ be a regular category. The following conditions are equivalent:
\begin{enumerate}
\item[(i)] $\CC$ is a Goursat category;
\item[(ii)] for any commutative cube
\begin{equation}\label{cube}
\vcenter{\xymatrix@C=30pt@R=20pt{
    X \times_Y Z \ar@<-2pt>[dd] \ar@<-2pt>[dr] \ar@{.>}[rr]^-{\lambda} & & U \times_W V \ar@<-2pt>@{-->}[dd] \ar@<-2pt>[dr] \\
    & Z \ar@<-2pt>[dd]_(.7){l} \ar@<-2pt>[ul] \ar@{>>}[rr]^(.2){\gamma} & & V \ar@<-2pt>[dd]_-h \ar@<-2pt>[ul] \\
    X \ar@<-2pt>[uu] \ar@<-2pt>[dr]_-{f} \ar@{-->>}[rr]^(.7){\alpha} & & U \ar@<-2pt>@{-->}[uu] \ar@<-2pt>@{-->}[dr]_-g \\
    & Y \ar@<-2pt>[uu]_(.3){k} \ar@<-2pt>[ul]_-{i} \ar@{>>}[rr]_-{\beta} & & W, \ar@<-2pt>[uu]\ar@<-2pt>@{-->}[ul]_-j }}
\end{equation}
where the left and right faces are pullbacks of split epimorphisms and $\alpha, \beta$ and $\gamma$ are regular epimorphisms (commuting also with the splittings), then the comparison morphism $\lambda \colon X \times _Y Z \to U \times_W V$ is also a regular epimorphism;
\item[(iii)] for any commutative cube
\begin{equation}\label{generalcube}
\vcenter{\xymatrix@C=35pt@R=20pt{
    X \times_Y Z \ar@<-2pt>[dd]_(.6){\pi_X} \ar@<-2pt>[dr]_-{\pi_Z} \ar@{>>}[rr]^-{\delta} & & A \ar@<-2pt>@{-->}[dd] \ar@<-2pt>[dr] \\
    & Z \ar@<-2pt>[dd]_(.7){l} \ar@<-2pt>[ul]_(.4){\langle il, 1_Z \rangle} \ar@{>>}[rr]^(.3){\gamma} & & V \ar@<-2pt>[dd]_(.4){h} \ar@<-2pt>[ul] \\
    X \ar@<-2pt>[uu]_(0.3){\langle 1_X, kf \rangle} \ar@<-2pt>[dr]_-{f} \ar@{-->>}[rr]^(.7){\alpha} & & U \ar@<-2pt>@{-->}[uu] \ar@<-2pt>@{-->}[dr]_-g \\
    & Y \ar@<-2pt>[uu]_(.3){k} \ar@<-2pt>[ul]_-{i} \ar@{>>}[rr]_-{\beta} & & W, \ar@<-2pt>[uu]\ar@<-2pt>@{-->}[ul]_-j }}
\end{equation}
where the left face is a pullback of split epimorphisms, the right face is a commutative diagram of split epimorphisms and the horizontal arrows $\alpha, \beta, \gamma, \delta$ are regular epimorphisms (commuting also with the splittings), then the right face is a pullback.
\end{enumerate}
\end{theorem}
\begin{proof}
(i) $\Rightarrow$ (ii). Again, it suffices to give a proof in set-theoretical terms. Let $(u,v)\in U\times_W V$. Then there exist $x\in X$ and $z\in Z$ such that $\alpha(x)=u$ and $\gamma(z)=v$; thus $g\alpha(x)=h\gamma(z)$. We define a binary relation $R$ on $Y\times U\times V$ by: $$((y_1,u_1,v_1) , (y_2,u_2,v_2)) \in R$$ if $y_1=l(\bar{z}), u_1=\alpha(\bar{x}), y_2=f(\bar{x})$ and $v_2=\gamma(\bar{z})$, for some $\bar{x}$ and $\bar{z}$. Categorically, 
$R$ is defined as follows. Let $$A=(Y\times U\times V)\times (Y\times U\times V)\times (X\times Z)$$ and consider the following four equalisers, where $\pi_i$ denotes the $i$-th projection:
$$
\begin{array}{l}
    \xymatrix{E_1 \, \ar@{>->}[r]^{} & A \ar@<3pt>[r]^-{l\pi_8} \ar@<-3pt>[r]_-{\pi_1} & Y}  \vspace{3pt}\\
   \xymatrix{E_2 \, \ar@{>->}[r] & A \ar@<3pt>[r]^-{\alpha\pi_7} \ar@<-3pt>[r]_-{\pi_2} & U}  \vspace{3pt} \\
   \xymatrix{E_3 \,  \ar@{>->}[r] & A \ar@<3pt>[r]^-{f\pi_7} \ar@<-3pt>[r]_-{\pi_4} & Y}  \vspace{3pt} \\
   \xymatrix{E_4 \, \ar@{>->}[r] & A \ar@<3pt>[r]^-{\gamma\pi_8} \ar@<-3pt>[r]_-{\pi_6} & Y}
\end{array}
$$
Let $E=E_1\wedge E_2\wedge E_3\wedge E_4$ be their intersection (as subobjects of $A$), and define $R$ as the regular image given by
$$
\xymatrix@R=10pt@C=30pt{E \ar@{ >->}[r] \ar@{>>}[dr] & A \ar[r] & (Y\times U\times V)\times (Y\times U\times V) \\
    & R \ar@{ >->}[ur]}
$$
along the obvious projection $A\to (Y\times U\times V)\times (Y\times U\times V)$.

We have the following relations:\\
$\cdot$ $(f(x),\alpha(x), \gamma k f(x))\, R \,(f(x),\alpha(x), \gamma k f(x))$, for $\bar{x}=x$ and $\bar{z}=k f(x)$; \vspace{5pt}\\
$\cdot$ $(l(z), \alpha i f (x), \gamma(z))\, R \,(f(x),\alpha(x), \gamma k l(z))$, for $\bar{x}=i f (x)$ and  $\bar{z}=k l (z)$; \vspace{5pt} \\
$\cdot$ $(l(z), \alpha i l (z), \gamma(z))\, R \, (l(z), \alpha i l (z), \gamma(z))$, for $\bar{x}=i l (z)$ and $\bar{z}=z$.
\\
From $g\alpha(x)=h\gamma(z)$, we get that $\gamma k f(x)= \gamma k l(z)$ and $\alpha i f (x) = \alpha i l (z)$. So:
$$(f(x),\alpha(x), \gamma k f(x)) (1_{Y\times U \times V} \wedge R)R^{\circ} (1_{Y\times U \times V} \wedge R) (l(z), \alpha i l (z), \gamma(z)). $$
From Theorem~\ref{Goursat characterisations}(iii), we can conclude that
$$(f(x),\alpha(x), \gamma k f(x))\, RR \,(l(z), \alpha i l (z), \gamma(z)). $$
Then there exists $(y',u',v')$ such that 
$$(f(x),\alpha(x), \gamma k f(x)) R (y',u',v') R (l(z), \alpha i l (z), \gamma(z)).$$ From the definition of $R$ we can conclude that, for some $\bar{x},\bar{z}, \bar{\bar{x}}, \bar{\bar{z}}$, $l(\bar{z})=f(x)$, $\alpha(\bar{x})=\alpha(x)$, $f(\bar{x})=y'$, $\gamma(\bar{z})=v'$ and $l(\bar{\bar{z}})=y'$, $\alpha(\bar{\bar{x}})=u'$, $f(\bar{\bar{x}})=l(z)$, $\gamma(\bar{\bar{z}})=\gamma(z)$. Then, there exists $(\bar{x},\bar{\bar{z}})\in X\times _Y Z$, since $f(\bar{x})=y'=l(\bar{\bar{z}})$, such that $\lambda(\bar{x},\bar{\bar{z}})=(\alpha(\bar{x}), \gamma(\bar{\bar{z}}))=(\alpha(x),\gamma(z))=(u,v)$.

(ii) $\Rightarrow$ (i). We can consider the left and right faces in the cube~(\ref{cube}) to be kernel pairs of split epimorphisms. Then condition (ii) translates into the statement of Theorem~\ref{Goursat po}(ii).

(ii) $\Rightarrow$ (iii).
By assumption we know that the induced arrow $$\lambda \colon X \times_Y Z \to U \times_W V$$ in diagram \eqref{cube} is a regular epimorphism, and this implies that the unique induced arrow $c \colon A \to  U \times_W V$ such that $c \delta = \lambda$ (where $\delta$ is defined as in diagram \eqref{generalcube}) is a regular epimorphism as well. In order to show that $c$ is also a monomorphism it will suffice to show that $\Eq(\lambda) \leqslant \Eq(\delta)$, since one always has $\Eq(\delta) \leqslant \Eq(\lambda)$. \\
As a preliminary step, we first show that $\Eq(\pi_X) \wedge \Eq(\lambda) \leqslant \Eq(\delta)$. \\ For that, we define a new relation $P$ as follows:
if $$
\xymatrix@R=10pt@C=30pt{\Eq(\pi_X)\wedge \Eq(\delta) \ar@{ >->}[rr]^-{\langle p_1,p_2 \rangle} & & (X\times_Y Z)\times (X\times_Y Z)}
$$ is the intersection of the relations $\Eq(\pi_X)$ and $\Eq(\delta)$ on $X\times_Y Z$,
the relation $P$ is given by the monomorphism
$$
\xymatrix@R=10pt@C=30pt{P= \Eq(\pi_X)\wedge \Eq(\delta) \ar@{ >->}[rr]^-{\langle p_1,\pi_Z p_2 \rangle} & & (X\times_Y Z)\times Z.}
$$
In set-theoretical terms $P$ is the relation from $X \times_Y Z$ to $Z$ defined as follows: \\
\centerline{ $((a,b), c) \in P$ if and only if $((a,b) , (a,c)) \in \Eq(\delta)$,}
\noindent (remark that $(a,c) \in X \times_Y Z$ by definition).
As we have observed above, under the assumption (ii) $\mathcal C$ is necessarily a Goursat category, and then by applying Theorem~\ref{Goursat characterisations} to the relation $P$ we get the equality $PP^{\circ}PP^{\circ} = PP^{\circ}$.

Now, let $((x,z),(x,w))$ be an element in $\Eq(\pi_X) \wedge \Eq(\lambda)$. This implies that $f(x)= l(z)= l(w)$, and $(z,w) \in \Eq(\gamma)$; accordingly,
 $((il(z),z),(il(w),w))$ $\in \Eq(\delta)$. Then:\\
$\cdot$ $((x,z),z) \in P$; \\
$\cdot$ $((il(z), w),z) \in P$, since $l(z)=l(w)$ and $((il(z),z),(il(w),w)) \in \Eq(\delta)$, as observed above; \\
$\cdot$ $((il(w),w),w) \in P$;\\
$\cdot$ $((x,w),w) \in P$. \\
It follows that $((x,z), (x,w)) \in PP^{\circ}PP^{\circ} = PP^{\circ}$. There is then a $\theta$ such that $((x,z), \theta)\in P$ and $((x,w), \theta) \in P$. The fact that $\Eq(\delta)$ is an equivalence relation implies that $((x,z),(x, w)) \in \Eq(\delta)$, since $((x,z),(x, \theta)) \in \Eq(\delta)$ and $((x,w),(x, \theta)) \in \Eq(\delta)$. It follows that $\Eq(\pi_X) \wedge \Eq(\lambda) \leqslant \Eq(\delta)$.

Finally, to prove that $\Eq(\lambda) \leqslant \Eq(\delta)$, we
consider then an element $((x_1, z_1), (x_2, z_2)) \in \Eq(\lambda)$: this means that $f(x_1)= l(z_1)$, $f(x_2)= l(z_2)$ and $\lambda (x_1, z_1) = \lambda (x_2, z_2)$ (thus, in particular, $\alpha(x_1) = \alpha(x_2)$). We are going to show that $\delta( x_1,z_1) = \delta (x_2,z_2)$. To do so, we apply the Shifting Lemma to the following situation
\begin{equation}\label{shiftpart}
\xymatrix@C=70pt@R=30pt{(x_1,kf(x_1))  \ar@{-}[r]^-{\Eq(\pi_X)} \ar@{-}[d]^-{\Eq(\lambda)} \ar@(l,l)@{-}[d]_-{\Eq(\delta)} &
  (x_1,z_1)   \ar@{-}[d]_-{\Eq(\lambda)} \ar@(r,r)@{--}[d]^-{\Eq(\delta)} \\
    (x_2,kf(x_2))  \ar@{-}[r]_-{\Eq(\pi_X)} &(x_2,z_2) , }
    \end{equation}
    where the solid lines represent relations holding by assumption.
Note that all the elements $(x_1,z_1), (x_2,z_2), (x_1,kf(x_1)), (x_2,kf(x_2))$ are in $X\times_Y Z$. Since $\alpha(x_1)=\alpha(x_2)$, we have $$((x_1,kf(x_1)), (x_2,kf(x_2)))\in \Eq(\delta)$$ by the commutativity of the upward back face of \eqref{generalcube}; it follows that $((x_1,kf(x_1)), (x_2,kf(x_2)))\in \Eq(\lambda)$ since $\Eq(\delta)\leqslant \Eq(\lambda)$.
The inequality $\Eq(\pi_X) \wedge \Eq(\lambda) \leqslant \Eq(\delta)$ that we have already checked above
allows us to apply the Shifting Lemma in the situation described by the bold lines in \eqref{shiftpart}, so that the relation in the dashed line holds, i.e. $((x_1,z_1), (x_2,z_2))\in \Eq(\delta)$, as desired.

(iii) $\Rightarrow$ (ii). This implication easily follows by taking the (regular epimorphism, monomorphism) factorisation of the comparison morphism $\lambda$ given in diagram \eqref{cube}, say $\lambda=m\delta$. One then obtains a cube of the type \eqref{generalcube} which, by assumption, is such that the right face is a pullback. Consequently, $\lambda$ is isomorphic to the regular epimorphism $\delta$.
\end{proof}

In the last part of this section we give a characterisation of regular Goursat categories through the fibration of points. Thus, we add Goursat categories to the list of (many) algebraic categories characterised in these terms (see ~\cite{MCFPO, BB}).

A \emph{point} in a category $\mathcal C$ is a split epimorphism $f \colon X\rightarrow Y$ together with a fixed splitting $i \colon Y\rightarrow X$, usually depicted as $$\xymatrix@1{X \ar@<-2pt>[r]_-f & Y \ar@<-2pt>[l]_-i}.$$

The category of points in $\mathcal C$ is denoted by $\Pt (\CC)$. When $\CC$ has pullbacks of split epimorphisms, the functor sending a point to its codomain
$$
\begin{array}{rcl}
    \Pt (\CC) & \to & \CC \\
    \xymatrix@1{U \ar@<-2pt>[r]_-g & W \ar@<-2pt>[l]_-j} & \mapsto & W
\end{array}
$$
is a fibration, called the \emph{fibration of points} \cite{NEKEAC}. Given a morphism $\beta\colon Y \to W$, the change-of-base functor with respect to this fibration is denoted by $\beta^*\colon \Pt_W(\CC) \to \Pt_Y(\CC)$. If $\CC$ has, moreover, pushouts along split monomorphisms, then any pullback functor $\beta^*$ has a left adjoint
$$
\begin{array}{lrcl}
    \beta_!: & \Pt_Y(\CC) & \to & \Pt_W(\CC), \\
    & \xymatrix@1{X \ar@<-2pt>[r]_-f & Y \ar@<-2pt>[l]_-i} & \mapsto & \xymatrix@1{\beta_!(X) \ar@<-2pt>[r]_-{\beta_!(f)} & W \ar@<-2pt>[l]_-{\beta_!(i)}}
\end{array}
$$
where
$(\beta_!(X), \beta_!(f), \beta_!(i)) \in \Pt_W(\CC)$ is determined by the right hand part of the following pushout:
$$
\xymatrix{ X \ar[r]^-{\overline{\beta}} & \beta_!(X)\\
Y \ar[r]_{\beta} \ar[u]^{i}& {W.} \ar[u]_{\beta_!(i)}}
$$
Observe that the arrow $\overline{\beta}$ in the diagram above is a regular epimorphism whenever $\beta$ is a regular epimorphism.

\begin{theorem}
\label{corollary Goursat cube}
Let $\mathcal C$ be a regular category with pushouts along split mono-morphisms. Then the following conditions are equivalent:
\begin{enumerate}
\item[(i)] $\mathcal C$ is a Goursat category;
\item[(ii)] for any regular epimorphism $\beta \colon Y \to W$ in $\mathcal C$ the functor \\ $\beta_{!} \colon \mathsf{Pt}_Y (\mathcal C) \to \mathsf{Pt}_{W} (\mathcal C)$ preserves binary products;
\item[(iii)] for any commutative square $$\xymatrix{
  X \ar@{>>}[r]^-{\alpha} \ar@<-2pt>[d]_-f & U  \ar@<-2pt>[d]_-g \\
  Y \ar@{>>}[r]_-{\beta} \ar@<-2pt>[u]_-i & W \ar@<-2pt>[u]_-j }
$$
where $f$ and $g$ are split epimorphisms and $\alpha$ and $\beta$ are regular epimorphisms (commuting also with the splittings), the Beck-Chevalley condition holds: there is a functor isomorphism $\alpha_! f^* \cong g^*\beta_!$.
\end{enumerate}
\end{theorem}
\begin{proof}
(i) $\Rightarrow$ (iii).
Given a point $(Z,l,k)$ over $Y$ consider the pullback defining $f^*(Z,l,k)$:
\begin{equation}\label{pullback}
\vcenter{\xymatrix@=25pt{
  {X \times_Y Z} \ar@<-2pt>[r]_-{{\pi}_Z} \ar@<-2pt>[d]_-{\pi_X} & Z  \ar@<-2pt>[d]_-l  \ar@<-2pt>[l] \\
  X \ar@<-2pt>[r]_-{f} \ar@<-2pt>[u] & {Y.} \ar@<-2pt>[u]_-k \ar@<-2pt>[l]_{i} }}
\end{equation}
The following commutative cube is obtained by applying the functor $\alpha_!$ and $\beta_!$ to the points $f^*(Z,l,k)$ (over $X$) and $(Z,l,k)$ (over $Y$), respectively:
$$
\xymatrix@C=35pt@R=20pt{
    X \times_Y Z \ar@<-2pt>[dd]_(.6){\pi_X} \ar@<-2pt>[dr]_-{\pi_Z} \ar@{>>}[rr]^-{\overline{\alpha}} & & \alpha_!(X\times_Y Z) \ar@<-2pt>@{-->}[dd] \ar@<-2pt>[dr] \\
    & Z \ar@<-2pt>[dd]_(.7){l} \ar@<-2pt>[ul] \ar@{>>}[rr]^(.3){\overline{\beta}} & & \beta_!(Z) \ar@<-2pt>[dd] \ar@<-2pt>[ul] \\
    X \ar@<-2pt>[uu] \ar@<-2pt>[dr]_-{f} \ar@{-->>}[rr]^(.7){\alpha} & & U \ar@<-2pt>@{-->}[uu] \ar@<-2pt>@{-->}[dr]_-g \\
    & Y \ar@<-2pt>[uu]_(.3){k} \ar@<-2pt>[ul]_-{i} \ar@{>>}[rr]_-{\beta} & & W. \ar@<-2pt>[uu]\ar@<-2pt>@{-->}[ul]_-j }
$$
This diagram is of the form (\ref{generalcube}). Consequently, the right face is a pullback by Theorem~\ref{Goursat cube}(iii), so that $g^*\beta_!\cong \alpha_! f^*$.

(iii) $\Rightarrow$ (ii). Let $(X,f,i)$ and $(Z,l,k)$ be points over $Y$ and consider their product in $\Pt_Y(\CC)$, which is given by the pullback \eqref{pullback}. We take its image through the functor $\beta_! \colon \Pt_Y(\CC) \rightarrow \Pt_W(\CC)$:
\begin{equation}\label{imagebeta}
\vcenter{\xymatrix@C=35pt@R=20pt{
    X \times_Y Z \ar@<-2pt>[dd]_(.6){\pi_X} \ar@<-2pt>[dr]_-{\pi_Z} \ar@{>>}[rr]^-{\overline{\beta}} & & \beta_!(X\times_Y Z) \ar@<-2pt>@{-->}[dd] \ar@<-2pt>[dr] \\
    & Z \ar@<-2pt>[dd]_(.7){l} \ar@<-2pt>[ul] \ar@{>>}[rr] & & \beta_!(Z) \ar@<-2pt>[dd] \ar@<-2pt>[ul] \\
    X \ar@<-2pt>[uu] \ar@<-2pt>[dr]_-{f} \ar@{-->>}[rr] & & \beta_!(X) \ar@<-2pt>@{-->}[uu] \ar@<-2pt>@{-->}[dr] \\
    & Y \ar@<-2pt>[uu]_(.3){k} \ar@<-2pt>[ul]_-{i} \ar@{>>}[rr]_-{\beta} & & W. \ar@<-2pt>[uu]\ar@<-2pt>@{-->}[ul] }}
\end{equation}
By applying the assumption to the bottom commutative face, we conclude that the right face is a pullback, i.e. $\beta_!$ preserves binary products.

(ii) $\Rightarrow$ (i). Consider the diagram \eqref{cube}, where $\alpha, \beta, \gamma$ are assumed to be regular epimorphisms, and let us show that the induced arrow $\lambda$ is also a regular epimorphism.
The image of the points over $Y$ of the left face of \eqref{cube} by the functor $\beta_{!}$ determines the commutative diagram \eqref{imagebeta}. One then obtains the following commutative diagram
\begin{equation}\label{comparison}
\vcenter{\xymatrix@C=35pt@R=20pt{
   \beta_{!} ( X \times_Y Z ) \ar@<-2pt>[dd] \ar@<-2pt>[dr] \ar@{-->}[rr]^-{\phi} & &  U \times_W V  \ar@<-2pt>@{-->}[dd] \ar@<-2pt>[dr] \\
    & \beta_{!}(Z) \ar@<-2pt>[dd] \ar@<-2pt>[ul]\ar@{>>}[rr]^(.3){\tau} & & V \ar@<-2pt>[dd]_(.4){} \ar@<-2pt>[ul] \\
    \beta_{!}(X) \ar@<-2pt>[uu] \ar@<-2pt>[dr] \ar@{-->>}[rr]^(.7){\sigma} & &  U  \ar@<-2pt>@{-->}[uu] \ar@<-2pt>@{-->}[dr] \\
    & W \ar@<-2pt>[uu] \ar@<-2pt>[ul] \ar@{=}[rr]_-{1_W} & & W, \ar@<-2pt>[uu]\ar@<-2pt>@{-->}[ul]}}
\end{equation}
where the arrows $\sigma, \tau$ and $\phi$ are induced by the universal properties of the pushouts defining $\beta_{!}(X), \beta_{!}(Z)$ and $\beta_{!} ( X \times_Y Z )$, respectively, and $\phi \overline{\beta} = \lambda$. The fact that $\alpha$ and $\gamma$ are regular epimorphisms implies that $\sigma$ and $\tau$ are regular epimorphisms, while the assumption guarantees that the left face in the diagram \eqref{comparison} is a pullback in $\mathcal C$. It then follows that the induced arrow $\phi$ is a regular epimorphism as well (as a product of the regular epimorphisms $\sigma$ and $\tau$ in $\Pt_W(\CC)$), and so is then the arrow $\lambda \colon X \times_Y Z \to U \times_W V $. This shows that condition (ii) in Theorem \ref{Goursat cube} is satisfied, and $\mathcal C$ is then a Goursat category, as desired.
\end{proof}

\begin{remark}\label{ternary}  A variety of universal algebras is $3$-permutable when its algebraic theory has two ternary operations $r$ and $s$ such that the identities
$r(x,y,y)=x$, $r(x,x,y)=s(x,y,y)$ and $s(x,x,y)=y$ hold \cite{Hagemann-Mitschke}. We can prove the existence of such ternary operations by applying the property stated in Theorem~\ref{Goursat cube}(ii) to the following commutative cube
$$
\xymatrix@C=60pt@R=35pt{
    P \ar@<-2pt>[dd] \ar@<-2pt>[dr] \ar@{.>}[rr]^-{\lambda} & & \Eq(\nabla) \ar@<-2pt>@{-->}[dd] \ar@<-2pt>[dr] \\
    & 3X \ar@<-2pt>[dd]_(.3){1+\nabla} \ar@<-2pt>[ul] \ar@{>>}[rr]^(.2){\nabla+1} & & 2X \ar@<-2pt>[dd]_-{\nabla} \ar@<-2pt>[ul] \\
    3X \ar@<-2pt>[uu] \ar@<-2pt>[dr]_-{\nabla+1} \ar@{-->>}[rr]_(.8){1+\nabla} & & 2X \ar@<-2pt>@{-->}[uu] \ar@<-2pt>@{-->}[dr]_-{\nabla} \\
    & 2X \ar@<-2pt>[uu]_(.7){{\left\lgroup i_1 \; i_2 \right\rgroup}} \ar@<-2pt>[ul]_(.45){{\left\lgroup i_2 \; i_3 \right\rgroup}} \ar@{>>}[rr]_-{\nabla} & & X, \ar@<-2pt>[uu]_-{i_1} \ar@<-2pt>@{-->}[ul]_-{i_2} }
$$
where $X$ is the free algebra on one element, $kX$ denotes a $k$-indexed copower of $X$ and $i_j$ the canonical $j$-th injection, $P$ denotes the object part of the pullback defining the left face and  $\nabla = {\left\lgroup 1_X \;\; 1_X \right\rgroup} \colon 2X \to X$ is the codiagonal. By Theorem \ref{Goursat cube}(ii), the comparison morphism $\lambda$ is a surjective homomorphism. The terms $p_1(x,y)=x$ and $p_2(x,y)=y$ are such that $(p_1,p_2)\in \Eq(\nabla)$. Since $\lambda$ is surjective, there exist ternary terms, say $r$ and $s$, such that:\\
\begin{tabular}{l}
  $(r,s)\in P$, from which we deduce that $r(x,x,y)=s(x,y,y)$, and \\
  $\lambda(r,s)=(p_1,p_2)$, which gives $r(x,y,y)=x$ and $s(x,x,y)=y$.
\end{tabular}
\end{remark}

%%%%%%%%%%%%%%%%%%%%%%%%%%%%%%%%%%%%%%%  SECTION: REFLECTIONS FOR GOURSAT CATEGORIES  %%%%%%%%%%%%%%%%%%%%%%%%%%%%%%%%%
\section{Reflections for Goursat categories}
In this section we show that any (regular epi)-reflective subcategory $\XX$ of a regular Goursat category $\CC$ has the property that the reflector $I \colon \CC \to \XX$ preserves pullbacks of split epimorphisms along split epimorphisms, extending the same result which is known for regular Ma'tsev categories (Proposition $3.6$ in \cite{BR}). Consequently, every internal Galois pregroupoid of an extension is an internal groupoid.

\begin{proposition}
\label{I preserves pbs}
Consider a (regular epi)-reflective subcategory $\XX$ of a regular Goursat category $\CC$
$$
 \xymatrix@1{\CC \ar@<1ex>[r]^I \ar@{}[r]|\bot & \;\XX, \ar@<1ex>@{_(->}[l]^U}
$$
where $U \colon \XX \to \CC$ is a full inclusion. Then the reflector $I \colon \CC \to \XX$ preserves pullbacks of pairs of split epimorphisms.
\end{proposition}
\begin{proof}
Consider the following commutative diagram where the left face is a pullback of split epimorphisms, the right face is its image through the functor $UI \colon \CC \to \CC$ and $\eta$ denotes the unit of the adjunction:
$$\xymatrix@C=30pt@R=20pt{
    X \times_Y Z \ar@<-2pt>[dd]_-{\pi_X} \ar@<-2pt>[dr]_(.6){\pi_Z} \ar@{>>}[rr]^-{\eta_{X\times_Y Z}} & & UI(X \times_Y Z) \ar@<-2pt>@{-->}[dd] \ar@<-2pt>[dr] \\
    & Z \ar@<-2pt>[dd]_(.7){l} \ar@<-2pt>[ul] \ar@{>>}[rr]^(.2){\eta_Z} & & UI(Z) \ar@<-2pt>[dd] \ar@<-2pt>[ul] \\
    X \ar@<-2pt>[uu] \ar@<-2pt>[dr]_-{f} \ar@{-->>}[rr]^(.7){\eta_X} & & UI(X) \ar@<-2pt>@{-->}[uu] \ar@<-2pt>@{-->}[dr] \\
    & Y \ar@<-2pt>[uu]_(.3){k} \ar@<-2pt>[ul]_-{i} \ar@{>>}[rr]_-{\eta_Y} & & UI(Y). \ar@<-2pt>[uu]\ar@<-2pt>@{-->}[ul] }
$$
This commutative diagram verifies the conditions of the commutative diagram \eqref{generalcube} in Theorem \ref{Goursat cube}. It follows that the right face is a pullback, since $\mathcal C$ is a Goursat category. But $U \colon \XX \to \CC$ is a full inclusion that preserves and reflects pullbacks, and this completes the proof.
\end{proof}

We follow~\cite{JK} in calling a (regular epi)-reflective subcategory $\XX$ of an exact category $\CC$ a \emph{Birkhoff} subcategory if, moreover, $\XX$ is closed in $\CC$ under regular quotients.

\begin{corollary}
\label{I preserves pbs 2}
Consider a Birkhoff subcategory $\XX$ of an exact Goursat category $\CC$
$$
 \xymatrix@1{\CC \ar@<1ex>[r]^I \ar@{}[r]|\bot & \;\XX, \ar@<1ex>@{_(->}[l]^U}
$$
where $U \colon \XX \to \CC$ is a full inclusion. Then the reflector $I \colon \CC \to \XX$ preserves pullbacks of split epimorphisms along regular epimorphisms.
\end{corollary}
\begin{proof}
Let \begin{equation}\label{splitpullback}
\vcenter{\xymatrix{
  P \ar@{>>}[r]^-{p_2} \ar@<-2pt>[d]_-{p_1} & C  \ar@<-2pt>[d]_-g \\
  A \ar@{>>}[r]_-p \ar@<-2pt>[u] & B \ar@<-2pt>[u] }}
\end{equation}
be a pullback of a split epimorphism $g$ along a regular epimorphism $p$. By taking the kernel pairs $\Eq(p_2)$ of $p_2$ and $\Eq(p)$ of $p$ and then applying $UI$ one gets the commutative diagram
\begin{equation}\label{leftpullbacks}
\vcenter{\xymatrix@=45pt{
UI(\Eq(p_2)) \ar@<-2pt>[r] \ar@<2pt>[r]  \ar@<-2pt>[d]_-{} &   UI(P) \ar@{>>}[r]^-{UI(p_2)} \ar@<-2pt>[d]_-{UI(p_1)} \ar@{}[dr]|-{(\mathrm A)}  & UI(C)  \ar@<-2pt>[d]_-{UI(g)} \\
UI(\Eq(p)) \ar@<-2pt>[r] \ar@<2pt>[r] \ar@<-2pt>[u]  &  UI(A) \ar@{>>}[r]_-{UI(p)} \ar@<-2pt>[u] & UI(B). \ar@<-2pt>[u] }}
\end{equation}
The two (downward oriented) commutative squares on the left are pullbacks, since $UI$ preserves pullbacks of split epimorphisms along split epimorphisms. It suffices to prove that $(\mathrm{A})$ is a pullback, since $U$ reflects pullbacks. Now, by taking the regular image of $\Eq(p_2)$ along the regular epimorphism $\eta_P$ one gets an effective equivalence relation $\eta_P(\Eq(p_2))$ on $UI(P)$, since $\CC$ is an exact Goursat category. Moreover, $UI(p_2)$ is its coequaliser in $\CC$, since $\XX$ is stable in $\CC$ under regular quotients. Accordingly, $\eta_P(\Eq(p_2))\cong \Eq(UI(p_2))$.
Similarly, $\eta_A (\Eq(p))\cong \Eq(UI(p))$ yielding the following commutative diagram:
$$\xymatrix@=45pt{
UI(\Eq(p_2)) \ar@{.>>}[r]   \ar@<-2pt>[d]_-{} &\Eq(UI(p_2))  \ar@<-2pt>[d] \ar@<-2pt> [r]   \ar@<2pt> [r] \ar@{}[dr]|-{(\mathrm B)} &  UI(P)   \ar@<-2pt>[d]_-{UI(p_1)}  \\
UI(\Eq(p))  \ar@{.>>}[r]  \ar@<-2pt>[u]  &  \Eq(UI(p))   \ar@<-2pt> [r]   \ar@<2pt> [r] \ar@<-2pt>[u] & UI(A). \ar@<-2pt>[u] }
$$
The exterior rectangles are the left pullbacks in diagram \eqref{leftpullbacks} and the dotted arrows are regular epimorphisms by construction. By applying Proposition $4.1$ in \cite{GR3} we conclude that the squares $(\mathrm{B})$ are pullbacks. By the usually called Barr-Kock theorem \cite{EC} we conclude that $(\mathrm{A})$ is a pullback, as desired.
\end{proof}

We finally observe that the so-called internal Galois pregroupoids associated to an extension are always internal groupoids in the Goursat context.

First recall from \cite{GJ} that an {\em internal precategory} in a category $\mathcal C$  is a diagram of the form
$$\xymatrix@=28pt{{P_{2}} \ar@<2ex>[r]^{p_{1}}
\ar@<-2ex>[r]^{p_{2}}
  \ar@<0ex>[r]^{m} &
 P_{1} \ar@<2ex>[r]^{d_{1}} \ar@<-2ex>[r]^{d_{2}} & P_{0} \ar@<0ex>[l]_{s} }$$
 with
\begin{enumerate}
 \item $d_{1}  s = 1_{P_0} = d_{2}  s$;
 \item $d_2  p_1 = d_1 p_2$;
\item $d_{1}  {p}_1 = d_{1}  m$, $ d_{2}  {p}_2 =
d_{2}  m $.
\end{enumerate}
In other words, an internal precategory in $\mathcal C$ can be seen as what remains of the definition of an internal category in $\mathcal C$ when one cancels all references to pullbacks.

Every extension (= regular epimorphism) $f \colon A \to B$ in $\mathcal C$ induces the internal groupoid given by the equivalence relation determined by the kernel pair of $f$:
$$\xymatrix@=28pt{\Eq(f) \times_A  \Eq(f) \ar@<2ex>[r]^-{p_{1}}
\ar@<-2ex>[r]^-{p_{2}}
  \ar@<0ex>[r]^-{m} &
 \Eq(f) \ar@<2ex>[r]^-{d_{1}} \ar@<-2ex>[r]^-{d_{2}} & A. \ar@<0ex>[l]_-{s} }
\vspace{5pt}
$$
By applying the left adjoint $I \colon \mathcal C \to {\mathcal X}$ of the inclusion functor $U \colon {\mathcal X} \to \mathcal C $ to this equivalence relation one always obtains an internal precategory in $\mathcal X$:
$$
\xymatrix@=30pt{I(\Eq(f) \times_A  \Eq(f) ) \ar@<3ex>[r]^-{I (p_{1})}
\ar@<-3ex>[r]^-{I(p_{2})}
  \ar@<0ex>[r]^-{I(m)} &
 I(\Eq(f)) \ar@<3ex>[r]^-{I(d_{1})} \ar@<-3ex>[r]^-{I(d_{2})} & I(A). \ar@<0ex>[l]_-{I(s)} }
$$
This special kind of internal precategory is called the {\em internal Galois pregroupoid of $f$} (see
\cite{GJ} for more details), and is denoted by $\mathsf{Gal}(f)$. It turns out to always be an internal groupoid in the Goursat context:
\begin{corollary}
\label{Galois groupoid}
Consider a (regular epi)-reflective subcategory $\XX$ of a regular Goursat category $\CC$
$$
 \xymatrix@1{\CC \ar@<1ex>[r]^I \ar@{}[r]|\bot & \;\XX, \ar@<1ex>@{_(->}[l]^U}
$$
where $U \colon \XX \to \CC$ is a full inclusion. Given any extension $f \colon A \to B$ in $\mathcal C$, then the internal Galois pregroupoid $\mathsf{Gal}(f)$ is an internal groupoid.
\end{corollary}
\begin{proof}
An internal precategory is an internal groupoid when certain commutative squares of split epimorphisms are pullbacks (see~\cite{BB}, for instance). The result follows immediately from the fact that those pullbacks are preserved by $I$ (Proposition~\ref{I preserves pbs}).
\end{proof}


\begin{thebibliography}{}
\bibitem{EC} M. Barr, Exact Categories, in: \emph{Lecture Notes in Math.} 236, Springer  1-120 (1971)
\bibitem{BerBour} C. Berger and D. Bourn, \emph{Central reflections and nilpotency in exact Mal'cev categories}, Preprint, 2015, arXiv:1511.00824.
\bibitem{BB} F. Borceux and D. Bourn, \emph{Mal'cev, protomodular, homological and semi-abelian categories}, Kluwer,
(2004)
\bibitem{NEKEAC}D. Bourn, \emph{Normalization equivalence, kernel equivalence and affine categories}, Springer Lecture
Notes in Math. 1488 (1991) 43-62.
\bibitem{MCFPO}D. Bourn, \emph{Mal'cev categories and fibration of pointed objects}, Applied Categorical Structures 4 (1996) 307-327.
\bibitem{BG0} D. Bourn and M. Gran, \emph{Categorical Aspects of Modularity}. Galois Theory, Hopf Algebras and Semiabelian Categories, Fields Instit. Commun., 43, Amer. Math. Soc., Providence RI (2004) 77-100.
\bibitem{BG1} D. Bourn and M. Gran, \emph{Normal sections and direct product decompositions} Comm. Algebra. 32 (2004) no. 10, 3825-3842.
\bibitem{BJ} D. Bourn and G. Janelidze, \emph{Protomodularity, descent and semidirect products}, Theory Appl. Categ. Vol. 4 (2) (1998) 37-46.
\bibitem{BR} D. Bourn and D. Rodelo, \emph{Comprehensive factorization and $I$-central extensions}, J. Pure Appl. Algebra 216 (2012) 598-617.
\bibitem{CKP} A. Carboni, G.M. Kelly, M.C. Pedicchio, \emph{Some remarks on Maltsev and Goursat categories}, Appl. Categ. Structures 1, no.4 (1993) 385-421.
\bibitem{CLP} A. Carboni, J. Lambek, M.C. Pedicchio, \emph{Diagram chasing in Mal'cev categories}, J. Pure Appl. Algebra 69, no. 3 (1991) 271-284.
\bibitem{FM} R. Freese and R. McKenzie, Commutator theory for congruence modular varieties, Cambridge
University Press, Cambridge (1987).
\bibitem{GR3x3} M. Gran and D. Rodelo, \emph{A new characterisation of Goursat categories}, Appl. Categ. Structures Vol. 20 (2012) 229-238.
\bibitem{GRCuboid} M. Gran and D. Rodelo, \emph{The Cuboid Lemma and Mal'tsev categories}, Applied Categorical Structures 22(5) (2014) 805-816.
\bibitem{GR3} M. Gran and D. Rodelo, \emph{Some remarks on pullbacks in Gumm categories}, Texts on Mathematics, 46, University of Coimbra (2014) 125-136.
\bibitem{Gumm} H.P. Gumm, \emph{Geometrical methods in congruence modular algebras}, Mem. Amer. Math. Soc. 45 (1983), no. 286.
\bibitem{Hagemann-Mitschke} J.~Hagemann and A.~Mitschke, \emph{On {$n$}-permutable congruences}, Algebra
  Universalis \textbf{3} (1973), 8--12.
\bibitem{GJ} G. Janelidze, \emph{Precategories and Galois theory},  Lecture Notes Math. 1488, Springer (1991) 157-173.
\bibitem{JK} G. Janelidze and G.M. Kelly, \emph{Galois theory and a general notion of central extension}, J. Pure Appl. Algebra 97 (1994), no. 2, 135-161
\bibitem{JRVdL} Z. Janelidze, D. Rodelo, T. Van der Linden,  \emph{Hagemann's theorem for regular categories}, Journal of Homotopy and Related Structures 9(1) (2014) 55-66.
\bibitem{L} S. Lack, \emph{The $3$-by-$3$ lemma for regular Goursat categories}, Homology, Homotopy Appl., 6(1), (2004) 1-3.
\bibitem{Maltsev-Sbornik} A.~I. Mal'cev, \emph{On the general theory of algebraic systems}, Mat. Sbornik
  N. S. \textbf{35} (1954), no.~6, 3--20.
\bibitem{Mitschke} A.~Mitschke, \emph{Implication algebras are 3-permutable and 3-distributive},
  Algebra Universalis \textbf{1} (1971), 182--186.
\bibitem{Pedicchio} M. C. Pedicchio, \emph{A categorical approach to commutator theory}, J. Algebra, 177 (1995) 647-657.
\bibitem{Smith} J.D.H. Smith, \emph{Mal'cev Varieties}, Lecture Notes in Math. 554 (1976).
\end{thebibliography}
\end{document}